\documentclass[11pt,a4paper]{article}
\usepackage{amsmath}
\usepackage{amssymb}
\usepackage{amsthm}
\usepackage{amsfonts}
\usepackage{enumerate}
\usepackage{verbatim}
\usepackage{hyperref}
\usepackage{breakurl,url}
\usepackage{bm}
\usepackage{color}
\usepackage{mathtools} %for := \coloneqq, for parentheses
\usepackage{pstricks}
\usepackage{pst-node}
\usepackage{pst-plot}
\usepackage{algorithm}
\usepackage{algorithmic}
\usepackage{pst-poly}
\usepackage{caption,subcaption} %subfigures
\usepackage[margin=1.3in]{geometry}
\newtheorem{theorem}{Theorem}

\newtheorem{proposition}{Proposition}

\newtheorem{lemma}{Lemma}
\newtheorem{corollary}{Corollary}

\theoremstyle{definition}

\newtheorem{definition}{Definition}
\newtheorem{example}{Example}
\newtheorem{remark}{Remark}
\begin{document}

\title{The Unique Solvability Conditions for the Generalized Absolute Value Equations}

\author{
Shubham Kumar\footnote{
Department of Mathematics, PDPM-Indian Institute of Information Technology, Design and Manufacturing Jabalpur, M.P. India, e-mail: \texttt{shub.srma@gmail.com}}
\and Deepmala\footnote{
Department of Mathematics, Faculty of Natural Sciences, PDPM-Indian Institute of Information Technology, Design and Manufacturing Jabalpur, M.P. India, e-mail: \texttt{dmrai23@gmail.com}}
% \and  Milan Hlad\'{i}k\footnote{
%Charles University, Faculty  of  Mathematics  and  Physics,
%Department of Applied Mathematics,
%Malostransk\'e n\'am.~25, 11800, Prague, Czech Republic,
%e-mail: \texttt{hladik@kam.mff.cuni.cz}}
%\and Hossein Moosaei\footnote{
%Department of Informatics, Faculty of Science, Jan Evangelista Purkyně University, \'{U}st\'{i} nad Labem, Czech Republic, 
%e-mail: \texttt{Hossein.Moosaei@ujep.cz}}
}

\date{\today}
\maketitle

\begin{abstract}
	This paper investigates the conditions that guarantee unique solvability and unsolvability for the generalized absolute value equations (GAVE) given by $Ax - B \vert x \vert = b$. Further, these conditions are also valid to determine the unique solution of the generalized absolute value matrix equations (GAVME)  $AX - B \vert X  \vert =F$. Finally, certain aspects related to the solvability and unsolvability of the absolute value equations (AVE) have been deliberated upon.
%This paper investigates the conditions that guarantee unique solvability and unsolvability for the generalized absolute value equations (GAVE) given by $Ax - B \vert x \vert = b$. Further, these conditions are also valid to determine the unique solution of the generalized absolute value matrix equations (GAVME)  $AX - B \vert X  \vert =F$.
\end{abstract}

\bigskip
%\paragraph{MSC Codes.}   90C30, 90C26, 90C33, 65G40,  15A06

%%%%%%%%%%%%%%%%%%%%%%%%%%%%%%%%%%%%%%%%%%%%%%%%%%%%%%%%%%%%%%%
% INTRODUCTION
%%%%%%%%%%%%%%%%%%%%%%%%%%%%%%%%%%%%%%%%%%%%%%%%%%%%%%%%%%%%%%%
\section{Introduction}
We are considering the following GAVE
\begin{equation} \label{Equ1}
	Ax -B\vert x \vert =b,
\end{equation}
with $A,B \in R^{n\times n},$ $b \in R^{n}$ are given and $x \in \mathbb{R}^{n}$ is to determined.

The following AVE is a standard form of the GAVE (\ref{Equ1}), defined as
\begin{equation} \label{Equ2}
	Ax -\vert x \vert =b.
\end{equation}
For a matrix $A \in R^{n \times n}$ and a vector $x \in R^{n }$, $\vert A \vert$ and $\vert x \vert$ denotes the component-wise absolute value of the matrix and the vector, respectively.
The exploration of AVE draws inspiration from the concept of interval linear equations \cite{rohn2004theorem}. Recently, the GAVE gets the more attention due to its wide applications in the different fields of mathematics and engineering applications, for example, bimatrix games, linear complementarity problems, linear programming etc. Rohn \cite{rohn2004theorem} first presented the alternative theorem for the unique solvability of the GAVE. Further, many authors investigated the unique solvability conditions of the GAVE in \cite{mangasarian2007absolute,mangasarian2006absolute,prokopyev2009equivalent,rohn2009unique,rohn2014iterative,wu2021unique,wu2016unique,wu2018unique,wu2020note}. Certain conditions for the unique solvability of the GAVE (\ref{Equ1}) are derived from the analysis of the interval matrix, singular values, spectral radius, and norms of the matrices A and B involved in the equation. But calculating the norm, spectral radius, and singular value of the sparse and large matrix may be computational time so high. Some conditions have limited practical usage or apply to tough, practical problems. For example, GAVE has a unique solution if $\rho (A^{-1} B\bar{\Lambda}) <1,$ or $\rho ( B\bar{\Lambda}A^{-1}) <1,$ satisfy for each $\bar{\Lambda} \in [-I_{n}, I_{n}]$. To apply such a condition, we have to know all $\bar{\Lambda} \in [-I_{n}, I_{n}]$, which is not an easy task. The regularity of the interval matrix $[A - |B|, A + |B|]$ implies the unique solvability of the GAVE, but determining the regularity of interval matrices is recognized as an NP-hard problem \cite{poljak1993checking}.

In this article, based on the different matrix classes, we provide some new unique solvability conditions for GAVE. The unique solvability conditions of GAVE (\ref{Equ1}) are also helpful for obtaining the unique solvability conditions for the LCP and the GAVME \cite{dehghan2020matrix,kumar2023note,kumar2022note,li2016note}.
Before designing an algorithm to solve a GAVE system, it is crucial to ensure the existence of a unique solution. Ensuring the existence of a unique solution guarantees that the algorithm will produce meaningful and accurate results. Therefore, it is essential to establish the solvability conditions for GAVE as a prerequisite for the algorithm design. In the existing literature, we introduce additional conditions that contribute to establishing the unique solvability of the GAVE.

%%%%%%%%%%%%%%%%%%%%%%%%%%%%%%%%%%%%%%%%%%%%%%%%%%%%%%%%%%%%%%%%%%%
\paragraph{Notation.}
We use $\sigma_{max}(A)$, $\sigma_{min}(A)$ and $\rho(A)$ for maximum singular value, minimum singular value and spectral radius of given matrix A, respectively. $|| . ||$ denotes the matrix norm. $A^{T}$ denotes the transpose of A and $I_{n}$ denotes identity matrix of size $n \times n.$ Given $x \in R^{n},$ the notation $sign(x)$ represents a vector whose elements are assigned values of 1, 0, or -1 based on whether the corresponding element of x is positive, zero, or negative. The diagonal matrix $D' = diag(sign(x))$ represents a diagonal matrix associated with the vector $sign(x)$.

%%%%%%%%%%%%%%%%%%%%%%%%%%%%%%%%%%%%%%%%%%%%%%%%%%%%%%%%%%%%%%%%%%%
\paragraph{Definitions.}
For $A,~ \Delta \in R^{n \times n}$, $\Delta \geq 0$, the interval matrix $\mathbb{I_{A}} = [A - \Delta, A + \Delta]$ is regular if each $K \in \mathbb{I_{A}}$ is invertible, otherwise $\mathbb{I_{A}}$ is singular. If  $A,~ \Delta$ are symmetric, then $\mathbb{I_{A}}$ is known as a symmetric interval matrix. A symmetric interval matrix is positive definite if each symmetric matrix $K \in \mathbb{I_{A}}$ is positive definite.

A matrix $A \in R^{n \times n}$ is categorized as a nonsingular M-matrix if and only if a matrix $\Delta \geq 0$ exists such that $A = \gamma I_{n} - \Delta$ where $\gamma > \rho(\Delta).$ $A \in R^{n \times n}$ is an H-matrix if and only if its comparison matrix is a nonsingular M-matrix. The signature of a symmetric matrix comprises a triplet representing the counts of its positive, negative, and zero eigenvalues.

%%%%%%%%%%%%%%%%%%%%%%%%%%%%%%%%%%%%%%%%%%%%%%%%%%%%%%%%%%%%%%%%%%%
\section{Generalized Absolute Value Equations}
Within this section, we have derived conditions that ensure the unique solvability of the GAVE.
\begin{definition} \cite{horn2012matrix}
	A matrix $A= (a_{ij}) \in R^{n \times n}$ is a strictly diagonally dominant (SDD) if and only if 
	$\vert a_{ii} \vert >  \sum_{j \ne i} \vert a_{ij} \vert,$ for $i \in \mathbb{N} = \{1,2,...,n\}.$

	%	if $\vert a_{ii} \vert >  \sum_{j \ne i} \vert a_{ij} \vert,$ for $i \in \mathbb{N} = \{1,2,...,n\}$ then matrix $A= (a_{ij}) \in R^{n \times n}$ is a strictly diagonally dominant (SDD).	

	%A matrix A is said to be SDD iff sum of all the row except diagonal element not exceed the absolute value of diagonal element. 
\end{definition} 

%%%%%%%%%%%%%%%%%%%%%%%%%%%%%%%%%%%%%%%%%%%%%%%%%%%%%%%%%%%%%%%%%%%%%%%%%%%%%%%%%%%%%%%%
\begin{theorem}\label{lemma0} \cite{achache2021unique}
	The GAVE(\ref{Equ1}) has exactly one solution if the invertible matrix A satisfies the norm condition $\vert \vert A^{-1}B \vert \vert < 1.$
\end{theorem}
%%%%%%%%%%%%%%%%%%%%%%%%%%%%%%%%%%%%%%%%%%%%%%%%%%%%%%%%%%%%%%%%%%%%%%%%%%%%%%%%%%%%%%%%

\begin{lemma} \label{lemma1} \cite{hu1982estimates}	
	If matrix $A \in R^{n \times n}$ is an SDD, then the following inequality holds for arbitrary matrix $B \in \mathbb{R}^{n \times n}$:	
	\begin{equation}
		\vert \vert A^{-1}B \vert \vert_{\infty} \le max_{i \in \mathbb{N}} \frac {\sum \vert b_{ij} \vert} {\vert a_{ii} \vert - \sum_{j \ne i} \vert a_{ij} \vert}
	\end{equation}	
	
	%	If matrix $A \in R^{n \times n}$ is an SDD, then for arbitrary matrix $B \in R^{n \times n}$, following inequlity hold:
	%\	\begin{equation}
	%\vert \vert A^{-1}B \vert \vert_{\infty} \le max_{i \in \mathbb{N}} \frac {\sum \vert b_{ij} \vert} {\vert a_{ii} \vert - %\sum_{j \ne i} \vert a_{ij} \vert}
	%\end{equation}	
\end{lemma} 

%%%%%%%%%%%%%%%%%%%%%%%%%%%%%%%%%%%%%%%%%%%%%%%%%%%%%%%%%%%%%%%%%%%%%%%%%%%%%%%%%%%%%%%%

\begin{theorem}\label{Thm Main}
	If the matrix $A,B \in R^{n \times n}$ satisfies
	\begin{equation} \label{1}
		\vert a_{ii} \vert > \sum \vert b_{ij} \vert + \sum_{j \ne i} \vert a_{ij} \vert, ~for~ i \in \mathbb{N},
	\end{equation}
	then GAVE (\ref{Equ1}) has a unique solution for any b.
\end{theorem}

\begin{proof}		
	Since	
	\begin{equation*}
		\vert a_{ii} \vert > \sum \vert b_{ij} \vert + \sum_{j \ne i} \vert a_{ij} \vert, ~for~ i \in \mathbb{N},
	\end{equation*}
	$\implies$ 
	\begin{equation*}
		\vert a_{ii} \vert > \sum_{j \ne i} \vert a_{ij} \vert, ~for~ i \in \mathbb{N},
	\end{equation*}	
	$\implies$ Matrix A is SDD.	
	
	To prove our result we will use Theorem \ref{lemma0} and for that we will show that $\vert \vert A^{-1}B \vert \vert < 1.$
	
	Since matrix A is SDD, then by Lemma \ref{lemma1} we have to show that 
	\begin{equation} \label{3}
		\vert \vert A^{-1}B \vert \vert_{\infty} \le max_{i \in \mathbb{N}} \frac {\sum \vert b_{ij} \vert} {\vert a_{ii} \vert - \sum_{j \ne i} \vert a_{ij} \vert} < 1.
	\end{equation}                
	
	By condition (\ref{1}), the inequality (\ref{3}) is hold. This implies that $\vert \vert A^{-1}B \vert \vert < 1,$ so by Theorem \ref{lemma0} GAVE (\ref{Equ1}) has a unique solution.
\end{proof}

%\begin{rem}
%In Theorem (\ref{Thm Main}), if we put $B=I$, then Theorem (\ref{Thm Main}) reduce to Theorem 2.1 in \cite{Shi Liang Wu 2016}, which is the main result of \cite{Shi Liang Wu 2016}.
%\end{rem}

%%%%%%%%%%%%%%%%%%%%%%%%%%%%%%%%%%%%%%%%%%%%%%%%%%%%%%%%%%%%%%%%%%%%%%%%%%%%%%%%%%%%%%%%

%We have the following result for the unique solvability of the AVE.
\begin{theorem}\label{Thm gave1} %\cite{ Lotfi 2013, Jiri Rohn 2009i, Shi-Liang Wu 2019, Shi Liang Wu 2021 GAVE} 
	If either of the following conditions satisfy then GAVE (\ref{Equ1}) has exactly one solution for each b:\\
	(i)\cite{rohn2009unique}  $\sigma_{max}( \vert B \vert) < \sigma_{min}(A);$\\
	(ii)\cite{wu2020note} $\sigma_{max}(B) < \sigma_{min}(A);$ \\
	(iii)\cite{achache2021unique}  $A^{T}A-\vert \vert B \vert \vert_{2}^{2}I_{n}$ is a positive definite matrix; \\
	(iv) \cite{lotfi2013note} $A^{T}A-\vert \vert (|B|) \vert \vert_{2}^{2}I_{n}$ is a positive definite matrix.
	%	(v)\cite{Shi-Liang Wu 2019}   $\rho ( A^{-1}B\bar{\Lambda}) <1,$ or $\rho ( B\bar{\Lambda}A^{-1}) <1,$ for arbitrary diagonal matrix $\bar{\Lambda} = diag(\bar{ \lambda_i})$ with $-1 \le \bar{\lambda_i} \le 1,$   where $\rho(.)$ denotes the spectral radius of a matrix;\\
	%	(vi)\cite{Shi-Liang Wu 2019}  $A+B\bar{\Lambda}$ is invertible for arbitrary diagonal matrix $\bar{\Lambda};$ \\
	%	(v)\cite{Lotfi 2013} interval matrix $[A - \vert B \vert, A + \vert B \vert]$ is regular.
\end{theorem}
%%%%%%%%%%%%%%%%%%%%%%%%%%%%%%%%%%%%%%%%%%%%%%%%%%%%%%%%%%%%%%%%%%%%%%%%%%%%%%%%%%%%%%%%

\begin{theorem}\label{Thm gave1i}  
	If either of the following conditions satisfy then GAVE (\ref{Equ1}) has exactly one solution for each b:\\
	(i) \cite{wu2021unique} $\rho ( A^{-1}B\bar{\Lambda}) <1,$ or $\rho ( B\bar{\Lambda}A^{-1}) <1,$ for each $\bar{\Lambda} \in [-I_{n}, I_{n}]$;\\
	(ii) \cite{wu2021unique} $A+B\bar{\Lambda}$ is invertible for each $\bar{\Lambda} \in [-I_{n}, I_{n}];$\\
	(iii) \cite{rohn2004theorem} interval matrix $[A - \vert B \vert, A + \vert B \vert]$ is regular.
\end{theorem}

%%%%%%%%%%%%%%%%%%%%%%%%%%%%%%%%%%%%%%%%%%%%%%%%%%%%%%%%%%%%%%%%%%%%%%%%%%%%%%%%%%%%%%%%

We consider the following example for the validity of the Theorem \ref{Thm Main} and compare our condition with the conditions of the Theorem \ref{Thm gave1} and Theorem \ref{Thm gave1i}.

\begin{example} \label{Ex1}
	Consider the following GAVE
	\begin{equation*}
		\begin{bmatrix}
			1.0200 &  0\\
			0  &  1.2003
		\end{bmatrix}
		\begin{bmatrix}
			x_1 \\
			x_2
		\end{bmatrix}
		-
		\begin{bmatrix}
			-0.8999 &  0.1009 \\
			-0.3788 & -0.7895
		\end{bmatrix}
		\begin{bmatrix}
			|x_1| \\
			|x_2|
		\end{bmatrix}
		=
		\begin{bmatrix}
			-0.5429 \\
			6.727
		\end{bmatrix}
	\end{equation*}	
	
	Here  $1.0200 > \vert -0.8999 \vert + 0.1999 + 0 = 1.0008$ and $1.2003 > \vert -0.3788 \vert +\vert -0.7895 \vert + 0 = 1.1683$. Clearly, the condition of Theorem \ref{Thm Main} is satisfying and it has a unique solution $x=[-2,3]^{T}$.
	
	But, the conditions of Theorem \ref{Thm gave1} are not satisfying. By simple calculation, $\sigma_{min}(A)=1.0200 \ngtr 1.0276 = \sigma_{max}(B),$  $\sigma_{min}(A)=1.0200 \ngtr 1.1022 = \sigma_{max}( \vert B \vert).$ Next, we have
	\begin{equation*}
		A^{T}A-\vert \vert B \vert \vert_{2}^{2}I_{n} =
		\begin{bmatrix}
			-0.015573 &  0\\
			0  &  0.38475
		\end{bmatrix}
	\end{equation*}
	is not a positive definite matrix, since it contains negative eigenvalue. Also
	\begin{equation*}
		A^{T}A-\vert \vert (\vert B \vert) \vert \vert_{2}^{2}I_{n} =
		\begin{bmatrix}
			-0.17438 &  0\\
			0  &  0.22594
		\end{bmatrix}
	\end{equation*}
	is not a positive definite matrix, since it contains negative eigenvalue.

	While to applying the conditions of Theorem \ref{Thm gave1i} are impractical and difficult, because to use conditions (i) and (ii), we need to identify all diagonal matrix $\bar{\Lambda}$  and identify all such $\bar{\Lambda}$ is not an easy task. To apply the condition (iii), we need to calculate all invertible matrices lie between interval matrix $[A - \vert B \vert, A + \vert B \vert]$, and the number of such invertible matrices may be uncountable. Moreover, the task of checking the regularity of interval matrices is known to be NP-hard.
\end{example}
%%%%%%%%%%%%%%%%%%%%%%%%%%%%%%%%%%%%%%%%%%%%%%%%%%%%%%%%%%%%%%%%%%%%%%%%%%%%%%%%%%%%%%%%

\begin{theorem}\label{Thm gave2}
	If any one of the following conditions satisfy then GAVE (\ref{Equ1}) has exactly one solution for each b:\\
	(i) \cite{wu2021unique} $\sigma_{max}(A^{-1}B) < 1$; \\
	(ii) \cite{rohn2014iterative} $\rho(\vert A^{-1}B \vert) < 1$.
\end{theorem}

% % % % % % % % % % % % % % % % % % % % % % % % % % % % % % % % % % % % % % % % % % % % % % % % %
%\begin{rem}
%	Since our condition \ref{1} does not has any relation with the singular value conditions and spectral radius conditions. So there exists some instance, where our condition \ref{1} not satisfy and other conditions are satisfy and vice versa. But, if we take an example, whereas conditions of Theorem \ref{Thm Main} and Theorem \ref{Thm gave2} are satisfying, then due to cost of computation to calculate singular value and spectral radius will be expensive in general.
%\end{rem}

\begin{remark}
	Since condition (\ref{1}) does not have a relation with the singular value conditions and spectral radius conditions. So there exists some instance where condition (\ref{1}) is not satisfied, and conditions of Theorem \ref{Thm gave2} are satisfied and vice-versa. But, if conditions of Theorem \ref{Thm Main} and Theorem \ref{Thm gave2} are satisfying for some cases, then due to the cost of computation to calculate the singular value and spectral radius may be expensive in general, see Exa. \ref{Ex1i}.
\end{remark}
% % % % % % % % % % % % % % % % % % % % % % % % % % % % % % % % % % % % % % % % % % % % % % % % %

%We consider the following example for comparing the executing time to the condition of Theorem \ref{Thm Main} and Theorem \ref{Thm gave2}.

We will now present an illustrative example to compare the CPU execution time with the conditions stated in Theorem \ref{Thm Main} and Theorem \ref{Thm gave2}~. This example aims to provide a practical demonstration of the differences between the two theorems.

\begin{example} \label{Ex1i} \cite{ali2022iterative} Consider the matrices $A = tridiag(-1,8,-1) \in R^{n \times n}$ and $B=I_{n}$ ($n \geq 3$).\\
	Corresponding GAVE $Ax - B \vert x \vert =b$ has a unique solution $x = (-1, 1, . . . , -1, 1) \in R^{n \times n}.$ Here we can easily verify the condition of Theorem \ref{Thm Main} for any $b \in R^{n}$.
	Since, for first and last row, we have $\vert 8 \vert >\vert -1 \vert + \vert 0 \vert + \vert 0 \vert + \vert 0 \vert + \vert 0 \vert + \vert 0 \vert + . . . +~\vert 0 \vert$. For each row (except first and last row) of matrices A and B, we have $\vert 8 \vert >\vert -1 \vert + \vert -1 \vert + \vert 0 \vert + \vert 0 \vert + \vert 0 \vert + \vert 0 \vert + . . . +~\vert 0\vert.$ 
	
	Moreover, we compare the CPU time of the condition of Theorem \ref{Thm Main} with the other two conditions of Theorem \ref{Thm gave2}, see \textbf{Table 1}. These calculations are executed in MATLAB (R 2016a) with Intel (R) Core (TM) i7-8700, CPU @ 3.20GHz with memory 8 GB.
	
	% Moreover, we compare the our condition CPU time with the other two conditions $\sigma_{max}(A^{-1}B) < 1$ \cite{Shi Liang Wu 2021 GAVE} and $\rho(\vert A^{-1}B \vert) < 1$ \cite{Jiri Rohn 2014}, see Table 1. For these calculations we used Intel (R) Core (TM) i7-8700, CPU @ 3.20GHz with memory 8 GB and in MATLAB (R 2016a).

	% % % % % % % % % % % % % % % % % % % % % % % % % % % % % % % % % % % % % % % % % % % % %
 
\begin{table} \label{Table 1}
	\begin{center}	
		\centering
		\caption{\textbf{CPU time taken (in seconds) by different conditions to solve Example \ref{Ex1i}}}
		\begin{tabular}	{ | p{0.6cm} | p{2.5cm}| p{2.0cm} | p{1.5cm}| p{1.8cm}| p{1.8cm}|  p{1.8cm}| p{1.8cm}| p{2.0cm}|} 
			\hline
			S.No. & Conditions & n (size of matrices) & 600   & 2000 & 3000 & 4000 & 5000 \\ 
			\hline \vspace{2mm}
			1. & $\rho(\vert A^{-1}B \vert) < 1$ & Time & 0.754401  & 21.611466 & 60.082111 & 147.111014 & 287.380045 \\
			\hline  \vspace{2mm}
			2. & $\sigma_{max}(A^{-1}B) < 1$& Time  & 0.110720  & 5.879430 & 14.753715 & 34.924410 &  67.968196 \\
			\hline \vspace{2mm}
			3. & Condition \ref{1} & Time  & 0.047611 & 2.479347 & 10.976085 & 28.159586 & 56.895144 \\
			\hline
		\end{tabular}
	\end{center} 	
\end{table}
\end{example} 

% % % % % % % % % % % % % % % % % % % % % % % % % % % % % % % % % % % % % % % % % % % % % % % % %

\begin{theorem} \label{Thm M-matrix}
	The GAVE (\ref{Equ1}) has a unique solution for any b, if $B \geq 0$ and  $A-B$ is an M-matrix.	
\end{theorem}
\begin{proof}
	Since matrix $A-B$ is an M-matrix, so $A-B = \gamma I_{n} - \Delta,$ where $\Delta \geq 0$ and $\alpha > \rho(\Delta).$ So $A = \gamma I_{n} - \Delta + B.$ \\
	Now, by simple calculation,
	\begin{align*}
		& \gamma I_{n} - \Delta + B - BD' &\\
		& = (\gamma I_{n} - \Delta) + B(I_{n} - D') & \\
		& \geq (\gamma I_{n} - \Delta).
	\end{align*}
	Hence, matrix ($\gamma I_{n} - \Delta + B - BD'$) is a non-singular M-matrix. Therefore, the associated linear system 
	($\gamma I_{n} - \Delta + B - BD'$)x = b has a unique solution for any b. This completes the proof.
\end{proof}
%%%%%%%%%%%%%%%%%%%%%%%%%%%%%%%%%%%%%%%%%%%%%%%%%%%%%%%%%%%%%%%%%%%%%%%%%%%%%%%%%%%%%%%%

\begin{corollary} \label{Thm H-matrix}
	The GAVE (\ref{Equ1}) has a unique solution for any b, if $B \geq 0$ and  $A-B$ is an H-matrix with positive diagonal elements.	
\end{corollary}
\begin{proof}
	By definition of the H-matrix and with the help of the Theorem \ref{Thm M-matrix}~, above result is hold.
\end{proof}
%%%%%%%%%%%%%%%%%%%%%%%%%%%%%%%%%%%%%%%%%%%%%%%%%%%%%%%%%%%%%%%%%%%%%%%%%%%%%%%%%%%%%%%%

\begin{theorem} \label{Thm positive definite}
	Let matrices $A$ and $B \geq 0$ be symmetric and $A$ is positive definite. Then GAVE (\ref{Equ1}) has a unique solution for any b if $x^{T}Ax - \vert x\vert^{T} B \vert x\vert > 0$ holds for every $x \ne 0$.	
\end{theorem}
\begin{proof}
	Since $A$ and $B \geq 0$ are  symmetric, then interval matrix $[A-B, A+B]$ is also symmetric.
	By [\cite{farhadsefat2012note}, Theorem 4.3], if $x^{T}Ax - \vert x\vert^{T} B \vert x\vert > 0$ holds for every $x \ne 0$ then symmetric interval matrix $[A-B, A+B]$ is positive definite. 
	Since A is symmetric positive definite and by [\cite{rohn1994positive}, Theorem 3], if $[A-B, A+B]$ is positive definite then $[A-B, A+B]$ is regular.
	
	Then by condition (iii) of the Theorem \ref{Thm gave1i}~, GAVE \ref{Equ1} has a unique solution.
\end{proof}

%%%%%%%%%%%%%%%%%%%%%%%%%%%%%%%%%%%%%%%%%%%%%%%%%%%%%%%%%%%%%%%%%%%%%%%%%%%%%%%%%%%%%%%%
\begin{theorem}\label{Thm Main GAVME}
	The conditions of the Theorems (\ref{Thm Main}~, \ref{Thm M-matrix}~, \ref{Thm positive definite}) and Corollary \ref{Thm H-matrix} are also applicable to determine the unique solution of the GAVME $AX - B \vert X \vert = F,$ where $A, B, F \in R^{n \times n}$.
	
	%	The GAVME $AX - B \vert X \vert = D,$ where $A, B \in R^{n \times n}$ has a unique solution for any  $D \in R^{n \times n}$ when condition \ref{1} is satisfy.
\end{theorem}
\begin{proof}
	For any matrix $ M = (m_{ij}) \in R^{n \times n},$ can be considered as the row of n numbers column vector $f_{k}$ where $k=1, 2, . . . , n.$ So GAVME can be written as the form of the n numbers of GAVE, i.e.,   $AX - B \vert X \vert = F$ is equivalent to the $Ax_{k} - B \vert x_{k} \vert = f_{k}.$ 
	
	Now, with the help of Theorem \ref{Thm Main}~, Theorem \ref{Thm M-matrix}~, Theorem \ref{Thm positive definite}) and Corollary \ref{Thm H-matrix}, GAVE $Ax_{k} - B \vert x_{k} \vert = f_{k}$ has a unique solution for all $f_{k}$. This completes the proof.
	%	The proof of the above theorem is simple and can be followed by the proof of Theorem 0.2 in \cite{Sharma 2022iv}.                                  
\end{proof}

% % % % % % % % % % % % % % % % % % % % % % % % % % % % % % % % % % % % % % % % % % % % % % % % %
\begin{remark}
	The conditions $\sigma_{min}(A) > \sigma_{max}( \vert B \vert)$ \cite{dehghan2020matrix} and $\sigma_{min}(A) > \sigma_{max}(B)$ \cite{xie2021unique} satisfied then GAVME has a unique solution for any $F \in R^{n \times n}$. Still, in some instances, these two conditions are not satisfied, while GAVME has a unique solution. Note that, here, condition (\ref{1}) of Theorem \ref{Thm Main GAVME} is satisfied.
\end{remark}
% % % % % % % % % % % % % % % % % % % % % % % % % % % % % % % % % % % % % % % % % % % % % % % % %
\begin{example} \label{Ex01}
	Consider the following GAVME
	\begin{equation*}
		\begin{bmatrix}
			2.2 &  0\\
			0  &  2.9
		\end{bmatrix}
		\begin{bmatrix}
			x_1 & x_2 \\
			x_3 & x_4
		\end{bmatrix}
		-
		\begin{bmatrix}
			-1 &  1 \\
			-1 & 1.5
		\end{bmatrix}
		\begin{bmatrix}
			|x_1| & |x_2| \\
			|x_3| & |x_4|
		\end{bmatrix}
		=
		\begin{bmatrix}
			2.6 & -10.4 \\
			-7.4 & -19.6
		\end{bmatrix}.
	\end{equation*}	
	
	The GAVME has a unique solution 
	X=
	\begin{equation*}
		\begin{bmatrix}
			3 &  -2\\
			-1  & -4
		\end{bmatrix}.
	\end{equation*}
\end{example} 
But conditions of \cite{dehghan2020matrix,xie2021unique} are not satisfying, since $\sigma_{min}(A)=2.2 \ngtr 2.2808 = \sigma_{max}(B)=  \sigma_{max}( \vert B \vert) $. Clearly, condition (\ref{1}) of Theorem \ref{Thm Main GAVME} is satisfying as follows $\vert 2.2 \vert> \vert -1 \vert +\vert 1 \vert + 0 = 2$ and $\vert 2.9 \vert > \vert -1 \vert +\vert 1.5 \vert + 0 = 2.5$.

%%%%%%%%%%%%%%%%%%%%%%%%%%%%%%%%%%%%%%%%%%%%%%%%%%%%%%%%%%%%%%%%%%%%%%%%%%%%%%%%%%%%%%%%

\begin{theorem}\label{Thm No sol} 
	The GAVE (\ref{Equ1}) has no solution, if for non-singular matrix B,  $0 \ne B^{-1}b \geq 0$ and
	\begin{equation} \label{11111}
		\sigma_{max}(A) < \sigma_{min}(B).
	\end{equation}
\end{theorem}
\begin{proof}  Since $\vert \vert A \vert \vert_{2} = \sigma_{max}(A)$ and  $\vert \vert B^{-1} \vert \vert_{2} = 1/\sigma_{min}(B).$ \\
	Let's assume that GAVE (\ref{Equ1}) has a non-zero solution x. The GAVE can be written as
	$B^{-1}Ax - \vert x \vert = B^{-1}b$. Since  $B^{-1}b \geq 0,$ so 	$B^{-1}Ax \geq \vert x \vert.$ Now taking the 2-norm, we get $ \vert \vert x \vert \vert_{2} \leq \vert \vert B^{-1}Ax \vert \vert_{2} \leq \vert \vert B^{-1} \vert \vert_{2} . \vert \vert A \vert \vert_{2} . \vert \vert x \vert \vert_{2} < \vert \vert x \vert \vert_{2}.$
	This is a contradiction, so GAVE (\ref{Equ1}) has no solution.	
\end{proof}
%%%%%%%%%%%%%%%%%%%%%%%%%%%%%%%%%%%%%%%%%%%%%%%%%%%%%%%%%%%%%%%%%%%%%%%%%%%%%%%%%%%%%%%%

We have the following consequences of the Theorem \ref{Thm No sol}.
\begin{proposition}\label{Thm No sol 2} 
	The GAVE (\ref{Equ1}) has no solution if either of the following conditions holds.\\
	(i) If 0 is not an eigenvalue of A and B, $0 \ne B^{-1}b \geq 0$ and 
	\begin{equation} \label{11}
		\sigma_{min}(A^{-1}B) > 1.
	\end{equation}
	(ii)  If 0 is not an eigenvalue of B, $0 \ne B^{-1}b \geq 0$ and 
	\begin{equation} \label{12}
		\sigma_{max}(B^{-1}A) < 1.
	\end{equation}
\end{proposition}
\begin{proof}
	By singular value properties of matrix, we have $\sigma_{max}(A) < \sigma_{min}(B)$  $\implies$ $1/\sigma_{min}(A^{-1}) < \sigma_{min}(B)$ $\implies$ $1 < \sigma_{min}(A^{-1}).\sigma_{min}(B)$ $\leq$ $\sigma_{min}(A^{-1}B).$ By Theorem \ref{Thm No sol} if $1 < \sigma_{min}(A^{-1}B)$ then GAVE (\ref{Equ1}) has no solution.\\
	Similar way, $\sigma_{max}(A) < \sigma_{min}(B)$  $\implies$  $\sigma_{max}(B^{-1})\sigma_{max}(A) < 1$ and for any two matrices  $A, B \in R^{n \times n},$ we have $\sigma_{max}(AB) \leq \sigma_{max}(A).\sigma_{max}(B) \cite{horn2012matrix}.$ Combine the last two inequalities, we have $\sigma_{max}(B^{-1}A) < 1$. Then by Theorem \ref{Thm No sol}, GAVE (\ref{Equ1}) has no solution.
\end{proof}

%%%%%%%%%%%%%%%%%%%%%%%%%%%%%%%%%%%%%%%%%%%%%%%%%%%%%%%%%%%%%%%%%%%%%%%%%%%%%%%%%%%%%%%%
In support of our conditions, we are considering the following simple example.
\begin{example} \label{Ex No Sol}
	Let
	\begin{equation*}
		A =
		\begin{bmatrix}
			-3 &  2\\
			-5  &  1
		\end{bmatrix},
		~	B =
		\begin{bmatrix}
			-1 &  7\\
			8  &  2
		\end{bmatrix},
		~ and ~	b =
		\begin{bmatrix}
			1 \\
			5  
		\end{bmatrix}.
	\end{equation*}
	Clearly, 
	\begin{equation*}
		0 \ne	 B^{-1}b =
		\begin{bmatrix}
			0.5690 \\
			0.2241
		\end{bmatrix}
		\geq 0.
	\end{equation*}
	Further,\\
	\noindent (i) $\sigma_{max}(A) = 6.1401 < 6.9414 = \sigma_{min}(B).$ \\
	(ii) $\sigma_{max}(B^{-1}A) =  0.7501 < 1.$ \\
	(iii) $\sigma_{min}(A^{-1}B) = 1.3331 > 1.$	\\
	So, GAVE $Ax - B \vert x \vert = b = [1 ~~ 5]^{T}$ does not has a unique solution.
\end{example}

%%%%%%%%%%%%%%%%%%%%%%%%%%%%%%%%%%%%%%%%%%%%%%%%%%%%%%%%%%%%%%%%%%%%%%%%%%%%%%%%%%%%%%%%

\section{Some Remarks on the Absolute Value Equations}
In this section, we have examined the conditions that lead to the unique solvability of the AVE. 

\noindent First, we present the following necessary and sufficient condition for the AVE.
\begin{theorem} \label{Thm ave0}
	Let A be a symmetric matrix. The AVE (\ref{Equ2}) has a unique solution if and only if $A-I_{n}$ and $A+I_{n}$ have the same signature.
\end{theorem}
\begin{proof}
	Since, AVE (\ref{Equ2}) has a unique solution if and only if interval matrix $[A-I_{n}, A+I_{n}]$ is regular \cite{wu2018unique,zhang2009global}. By Proposition 2.5 of \cite{hladik2023properties}, interval matrix $[A-I_{n}, A+I_{n}]$ is regular if and only if $A-I_{n}$ and $A+I_{n}$ have the same signature. This complete the proof.
\end{proof}

%%%%%%%%%%%%%%%%%%%%%%%%%%%%%%%%%%%%%%%%%%%%%%%%%%%%%%%%%%%%%%%%%%%%%%%%%%%%%%%%%%%%%%%%

\begin{theorem}\label{Thm ave1} %\cite{Mangasarian 2006,Shi Liang Wu 2016,Shi Liang Wu 2018,Shi Liang Wu 2021 GAVE} 
	If either of the following conditions satisfy, then AVE (\ref{Equ2}) has exactly one solution:\\
	(i)  \cite{mangasarian2006absolute} $1 < \sigma_{min}(A)$;\\
	(ii) \cite{wu2021unique} $2 <  \sigma_{min}(A+I_{n});$ \\	
	%	(iii) \cite{Shi Liang Wu 2016}
	%	\begin{equation*}
	%		\vert a_{ii} \vert > 1 + \sum_{j \ne i} \vert a_{ij} \vert, ~for~ i \in \mathbb{N};
	%	\end{equation*}
	(iii)\cite{zhang2009global} interval matrix $\mathbb{I_{A}} = [A-I_{n},A+I_{n}]$ is regular;\\
	(iv)  \cite{wu2018unique} $A-I_{n}+2D$ is invertible for each $D \in [0,I_{n}];$ \\
	(v)  \cite{wu2018unique} $A+I_{n}-2D$ is invertible for each $D \in [0,I_{n}].$
	%   (vii) $\rho(\vert A^{-1} \vert) < 1.$
	%	(iii) \cite{Lotfi 2013}  $A'A-I$ is a positive definite matrix; \\
	%	(iii) \cite{Mohamaed 2021,Lotfi 2013}  $A'A- I$ is a P-matrix;\\
\end{theorem}

%%%%%%%%%%%%%%%%%%%%%%%%%%%%%%%%%%%%%%%%%%%%%%%%%%%%%%%%%%%%%%%%%%%%%%%%%%%%%%%%%%%%%%%%

\begin{remark}
	By Proposition 2.1 of \cite{zhang2009global} and Theorems 3.2-3.3 of \cite{wu2018unique}, the conditions (iii), (iv) and (v) of the Theorem \ref{Thm ave1} are equivalent.
	In Theorem \ref{Thm ave1}, the condition (i) implies the condition (iii), see \cite{rex1998sufficient} but the converse need not be true, see Exa. 2.3 in \cite{zhang2009global}. Condition (ii) implies condition (i) but is not necessarily converse; see Proposition \ref{Cor ave} and Proposition \ref{Thm ave1ii}.
	
	% Since all the matrices $A-I+2D$ and $A+I-2D$ are belongs to the set of regular interval matrix $\mathbb{I_{A}}$.
\end{remark}

% % % % % % % % % % % % % % % % % % % % % % % % % % % % % % % % % % % % % % % % % % % % %

In the Exa. (\ref{Ex2}), AVE (\ref{Equ2}) has a unique solution which is confirmed by the condition (i) of the Theorem \ref{Thm ave1}. In contrast, condition (ii) of the Theorem \ref{Thm ave1} is insufficient to determine the unique solvability of (\ref{Equ2}). These conditions may be revised in the future.
\begin{example} \label{Ex2}
	Let
	\begin{equation*}
		A =
		\begin{bmatrix}
			1.4878 &  -0.1142\\
			0.3203  &  -1.1589
		\end{bmatrix}.
	\end{equation*}
	Here $\sigma_{min}(A)= 1.0549>1,$ but $\sigma_{min}(A+I_{n})=0.1428 \ngtr 2$ and $\vert -1.1589 \vert  \ngtr 1 + 0.3203.$ %Also, $\vert \vert (A-I)^{-1} \vert \vert_{2} = 2.1551  \nless 2$ and $\vert \vert (A+I)^{-1} \vert \vert_{2} = 7.0021 \nless 2.$
	%So the conditions (ii) and (v) of the Theorem \ref{Thm ave1} need further investigations.
\end{example}

% % % % % % % % % % % % % % % % % % % % % % % % % % % % % % % % % % % % % % % % % % % % %
Wu et al. \cite{wu2021unique} show that the condition $2 <  \sigma_{min}(A+I_{n})$ is slighter stronger than the condition $1 < \sigma_{min}(A)$, but this fact is not correct. From Exa. \ref{Ex2}, we can make the following assumption.
\begin{proposition} \label{Cor ave}
	If the condition $\sigma_{min}(A) > 1$ is satisfy then condition $\sigma_{min}(A+I_{n}) > 2$ not necessarily hold for matrix $A \in R^{n \times n}$.
\end{proposition}
% % % % % % % % % % % % % % % % % % % % % % % % % % % % % % % % % % % % % % % % % % % % %

\begin{proposition} \label{Thm ave1ii}
	If the condition $\sigma_{min}(A+I_{n}) > 2$ is satisfy then the condition $\sigma_{min}(A) > 1$ also satisfy.
\end{proposition}
\begin{proof}
	Let $\sigma_{min}(A+I_{n}) > 2$ is hold, then	
	\begin{align*}
		2 & < \sigma_{min}(A+I_{n}) &\\
		& \leq \sigma_{min}(A) + \sigma_{min}(I_{n}) & \\
		& = \sigma_{min}(A) + 1 &\\
		& \implies \sigma_{min}(A) >1.
	\end{align*}
\end{proof}

Proposition \ref{Cor ave} and Proposition \ref{Thm ave1ii} establish a relationship between the conditions (i) and (ii) of the Theorem \ref{Thm ave1} for the given problem. Specifically, condition (ii) is shown to imply condition (i), meaning that if condition (ii) is satisfied, condition (i) must also be satisfied. However, the converse is not necessarily true, indicating that the fulfillment of condition (i) does not guarantee condition (ii).

%%%%%%%%%%%%%%%%%%%%%%%%%%%%%%%%%%%%%%%%%%%%%%%%%%%%%%%%%%%%%%%%%%%%%%%%%%%%%%%%%%%%%%%%

\begin{theorem}\label{Thm ave1i} %\cite{Mangasarian 2006,Shi Liang Wu 2016,Shi Liang Wu 2018,Shi Liang Wu 2021 GAVE} 
	If either of the following conditions satisfy, then AVE (\ref{Equ2}) has exactly one solution:\\
	(i)   \cite{wu2018unique} $\vert \vert (A-I_{n})^{-1} \vert \vert_{2} < 2 $ or $\vert \vert (A+I_{n})^{-1} \vert \vert_{2} < 2$; \\
	%	(iv)  Let matrix A is positive definite and H =  $\frac{1}{2}(A+A^{T})$ and K =  $\frac{1}{2}(A-A^{T})$ be its hermitian and skew Hermitian parts. If 1 is not an eigenvalue of H \cite{Shi Liang Wu 2016};\\
	(ii)  \cite{wu2016unique} If H is hermitian part of the positive definite matrix A and $H-I_{n}$ is invertible.
\end{theorem}

% % % % % % % % % % % % % % % % % % % % % % % % % % % % % % % % % % % % % % % % % % % % %

The conditions of the Theorem \ref{Thm ave1i} are not correct, see Exa. (\ref{Ex3}) and Exa. (\ref{Ex4}).
\begin{example} \label{Ex3} The conditions $\vert \vert (A+I_{n})^{-1} \vert \vert_{2} < 2$ and $\vert \vert (A-I_{n})^{-1} \vert \vert_{2} < 2$ are insufficient to determine the unique solvability of the AVE (\ref{Equ2}). Let consider 
	\begin{equation*}
		A =
		\begin{bmatrix}
			1 &  1\\
			1  &  1
		\end{bmatrix}.
	\end{equation*}	
	The AVE	(\ref{Equ2}) has three solutions, namely (-2,-1), (-4,3) and (6,-5). However, $\vert \vert (A+I_{n})^{-1} \vert \vert_{2} = 1 < 2$ and $\vert \vert (A-I_{n})^{-1} \vert \vert_{2} = 1 < 2.$
	
	% Let consider matrix A = $\frac{1}{2}I_{1 \times 1}$. Here $\vert \vert (A+I)^{-1} \vert \vert_{2} = 0.6667 < 2.$ But, AVE $0.5x - \vert x \vert = -2$ has two solution,  namely x =4 and x=-1.333, while AVE $0.5x - \vert x \vert = 2$ is infeasible.
	
	%	Now consider matrix A = $\frac{1}{5}I_{1 \times 1}$. Here $\vert \vert (A-I)^{-1} \vert \vert_{2} = 1.25 < 2.$ But, AVE $0.2x - \vert x \vert = 5$ is infeasible.
\end{example}
% % % % % % % % % % % % % % % % % % % % % % % % % % % % % % % % % % % % % % % % % % % % %
%\begin{rem}
%If the condition (ii) of the Theorem \ref{Thm ave1i} is satisfying, then in some instances, AVE (\ref{Equ2}) has a unique solution and for some instances AVE (\ref{Equ2}) does not has a unique solution. So it is an open question that, from which additional condition with matrix A, AVE (\ref{Equ2}) has a unique solution (or not a unique solution).
%\end{rem}

\begin{remark}
	If condition (i) of Theorem \ref{Thm ave1i} is satisfied, the solution of the AVE (\ref{Equ2}) can exhibit two possibilities:\\
	(i). Unique Solution: In some instances, the AVE (\ref{Equ2}) has a unique solution.\\
	(ii). Non-Unique Solution: However, in other instances, the AVE (\ref{Equ2}) does not have a unique solution.
	
	So here arise an open question. The open question is: What additional conditions on matrix A will determine whether the AVE (\ref{Equ2}) has a unique solution for each b or whether it does not have a unique solution for each b?
	
	%So here arise an open question that, from which additional condition with matrix A, AVE (\ref{Equ2}) can be a unique solution for each b (or not a unique solution for each b).
\end{remark}
% % % % % % % % % % % % % % % % % % % % % % % % % % % % % % % % % % % % % % % % % % % % %

\begin{example} \label{Ex4}
	For condition (ii) of the Theorem \ref{Thm ave1i}, consider the positive definite matrix A = 0.25, then H = 0.25. Clearly, $H-I_{n}$ is invertible. But AVE $ 0.25x - \vert x \vert = -5 $ has two solutions, namely $x = 6.66667$ and $x= -4$. 
\end{example}

% % % % % % % % % % % % % % % % % % % % % % % % % % % % % % % % % % % % % % % % % % % % %

\begin{proposition} \label{prop 1}
	The AVE (\ref{Equ2}) does not has a unique solution for all b if the interval matrix $[A-I_{n}, A+I_{n}]$ is singular.
\end{proposition}
\begin{proof}
	Let  interval matrix $[A-I_{n}, A+I_{n}]$ is singular then by the alternative theorem of \cite{rohn2004theorem}, AVE (\ref{Equ2}) does not has a unique solution.
\end{proof}

\begin{proposition} \label{prop 2}
	The AVE (\ref{Equ2}) does not has a unique solution for all b if the following conditions are hold: \\
	(i) $\vert Ax \vert  \leq \vert x \vert$ has a non-trivial solution;\\
	%(ii) for invertible matrix A, the condition $max_{j}(\vert A^{-1} \vert)_{jj} \geq 1$ hold, where $j \in \{1,2,. . . ,n\}$ and for matrix $A$, $A_{j}$ denote the $j^{th}$ column of A;\\
	(ii) $\sigma_{max}(A) \leq 1;$\\
	(iii) $I_{n}-A^{T}A$ is positive semi-definite;\\
	(iv) matrix $A$ has one eigenvalue is 0 or 1.
	%	(v) determinant of A is 1. 
\end{proposition}
\begin{proof}
	According to Theorem 2.1, Theorem 4.2 and Theorem 5.2 of \cite{rex1998sufficient}, the interval matrix $[A-I_{n}, A+I_{n}]$ is singular if one of the any conditions (i - iii) of the Proposition \ref{prop 2} holds. \\
	If matrix A has one eigenvalue 0 or 1, then $A$ or $A-I_{n}$ will be singular. So the interval matrix $[A-I_{n}, A+I_{n}]$ is singular.
	Then by the Proposition \ref{prop 1}~, AVE (\ref{Equ2}) does not have a unique solution for all b.
\end{proof}

% % % % % % % % % % % % % % % % % % % % % % % % % % % % % % % % % % % % % % % % % % % % %
%\section{Conclusion and future directions}
\section{Conclusions}
In this paper, based on the matrices classes, we discussed the unique solvability of the GAVE and AVE. Some previous conditions are not appropriate to check the unique solvability of the GAVE and AVE, and these conditions need further investigation. Moreover, we show that unique solvability conditions are also applicable to check the unique solvability of the GAVME. Given the established equivalence between the  AVE and the linear complementarity problem (LCP) and between the GAVE and the horizontal LCP, the outcomes regarding the unique solvability presented in this paper carry significant implications for both LCP and horizontal LCP.

% Due to the equivalence between AVE and linear complementarity problem (LCP) and GAVE and horizontal LCP, the unique solvability results in this paper are also useful to the LCP and horizontal LCP.

\paragraph{Acknowledgments.} 
The research work of Shubham Kumar was supported by the Ministry of Education, Government of India, through Graduate Aptitude Test in Engineering (GATE) fellowship registration No. MA19S43033021.

% % % % % % % % % % % % % % % % % % % % % % % % % % % % % % % % % % % % % % % % % % % % %
% % % % % % % % % % % % % % % % % % % % % % % % % % % % % % % % % % % % % % % % % % % % %
% % % % % % % % % % % % % % % % % % % % % % % % % % % % % % % % % % % % % % % % % % % % %

%%%%%%%%%%%%%%%%%%%%%%%%%%%%%%%%%%%%%%%%%%%%%%%%%%%%%%%%%%%%%%%
% REFERENCES
%%%%%%%%%%%%%%%%%%%%%%%%%%%%%%%%%%%%%%%%%%%%%%%%%%%%%%%%%%%%%%%

\bibliographystyle{abbrv}
\bibliography{gave_solv}

\end{document}